\documentclass[12pt]{article}

\usepackage{amssymb}%
\usepackage{amsmath}%
\usepackage{amsthm}%
\usepackage{xcolor}%

\allowdisplaybreaks%

{\theoremstyle{plain}%
 \newtheorem{theorem}{Theorem}

}
{\theoremstyle{remark}

}
{\theoremstyle{definition}

\newtheorem{example}{Example}
}

\begin{document}

\begin{center}
{\Large On the minimal polynomials of the arguments of dilogarithm ladders}
\end{center}

\begin{center}
{\textsc{John M. Campbell}} 

 \ 

\end{center}

\begin{abstract} 
 Letting $L_{n}(N, u)$ denote a polylogarithm ladder of weight $n$ and index $N$ with $u$ as an algebraic number, there is a rich history surrounding how 
 mathematical objects of this form can be constructed for a given weight or index. This raises questions as to what minimal polynomials for $u$ are permissible 
 in such constructions. Classical relations for the dilogarithm $\text{Li}_{2}$ provide families of weight-2 ladders in such a way so that the base equations 
 for $u$ consist of a fixed number of terms, and subsequent constructions for dilogarithm ladders rely on sporadic cases whereby $u$ is defined via a 
 cyclotomic equation, as in the supernumary ladders due to Abouzahra and Lewin. This motivates our construction of an infinite family of dilogarithm ladders 
 so as to obtain arbitrarily many terms with nonzero coefficients for the minimal polynomials for $u$. Our construction relies on a derivation of a 
 dilogarithm identity introduced by Khoi in 2014 via the Seifert volumes of manifolds obtained from the use of Dehn surgery. 
\end{abstract}


\vspace{0.1in}


\section{Introduction}\label{sectionIntro}
 The polylogarithm function $\text{Li}_{n}$ may be defined for $|z| \leq 1$ so that $\text{Li}_{n}(z) := \sum_{k=1}^{\infty} \frac{z^{k}}{k^n}$ and for a 
 given parameter $n$, with the $n = 2$ case yielding the dilogarithm. The dilogarithm function arises in deep areas of algebraic $K$-theory, differential 
 geometry, and mathematical physics. The study of $\text{Li}_{2}$ in its own right often leads to similarly deep areas within number theory. Lewin is one the 
 most pioneering figures in the history of the dilogarithm function, with reference to standard textbooks concerning this function 
 \cite{Lewin1981,Lewin1991text}. In this regard, Lewin was instrumental in the development of dilogarithm ladders. In this paper, we construct an infinite 
 family of dilogarithm ladders such that, in contrast to previously known ladder constructions, the minimal polynomials for the algebraic arguments for our 
 dilogarithm ladders contain arbitrarily many nonzero terms. 

 The definition of the term \emph{ladder} has evolved, with regard to the relevant literature on polylogarithms \cite[p.\ 6]{Lewin1991text}. According to 
 a 1987 reference from Abouzahra, Lewin, and Hongnian \cite{AbouzahraLewinXiao1987}, a polylogarithmic ladder of index $N$ and weight $n$ is a linear 
 combination of the form 
\begin{equation}\label{restrictdefinition}
 \frac{\text{Li}_{n}\left( u^{N} \right)}{N^{n-1}}
 - \left( \sum_{r} A_{r} \frac{\text{Li}_{n}\left( u^{r} \right)}{r^{n-1}} 
 + B_{0} \frac{\log^n u}{n!} + B_{2} \frac{\zeta(2) \log^{n-2}u}{(n-2)!} \right) 
\end{equation}
 for rational coefficients $A_{r}$ and $B_{0}$ and $B_{2}$, and where $u$ is algebraic over $\mathbb{Q}$, and where the sum in \eqref{restrictdefinition} 
 is over all positive divisors of $N$ apart from $N$. Furthermore, according to Abouzahra et al., this definition is such that \eqref{restrictdefinition} 
 necessarily vanishes if $n = 2$. However, the definition of the expression \emph{ladder} was subsequently generalized in the standard monograph on 
 structural properties of polylogarithms \cite{Lewin1991text}. According to this more general and standard definition \cite[p.\ 6]{Lewin1991text}, a 
 \emph{ladder} of weight (or order) $n$ and index $N$ and base $u$ is a linear combination of the form 
\begin{equation}\label{correctdefinition}
 L_{n}(N, u) := \frac{\text{Li}_{n}\left( u^{N} \right)}{N^{n-1}} - \left( \sum_{r = 0}^{N-1} \frac{ A_{r} \text{Li}_{n}\left( u^{r} \right) }{r^{n-1}} 
 + \frac{ A_{0} \log^{n}(u) }{n!} \right) 
\end{equation}
 for rational coefficients $A_{r}$ and for a value $u$ that is algebraic over $\mathbb{Q}$, and where the right-hand side of the equality in 
 \eqref{correctdefinition} is equal to a linear combination of the form 
\begin{equation}\label{Llinear}
 L_{n} = \sum_{m=2}^{n} \frac{D_{m} \zeta(m) \log^{n-m}(u) }{(n-m)!} 
\end{equation}
 for scalars $D_{m}$. For example, writing $u = \frac{\sqrt{3} - 1}{2}$, the relation 
\begin{equation}\label{counterexample}
 2 \text{Li}_{2}\left( u^3 \right) - 3 \text{Li}_{2}\left( u^2 \right) - 12 \text{Li}_{2}\left( u \right) - 3 \log^2 u = - 5 \zeta(2) 
\end{equation}
 due to Loxton in 1984 \cite{Loxton1984} is such that the left-hand side of \eqref{counterexample} is referred to as a \emph{ladder} by Lewin 
 \cite{Lewin1984} \cite[p.\ 22]{Lewin1991text}, but the linear combination on the left of \eqref{counterexample} does not satisfy the conditions of the 
 previous and more restrictive definition from Abouzahra et al.\ \cite{AbouzahraLewinXiao1987}. 

 If the numerical values of the form $D_{m}$ in \eqref{Llinear} are rational, then the ladder in \eqref{correctdefinition} is said to be \emph{valid}. 
 As in Lewin's text \cite[p.\ 6]{Lewin1991text}, we express that valid ladders are the only ladders of interest in the study of polylogarithms. For the 
 purposes of this paper, the only ladders to be considered are valid ladders. 

 Our paper is largely based on the following question, and the motivation surrounding this question is clarified via the background material in Section 
 \ref{sectionBackground}. How could we obtain dilogarithm ladders such that the minimal polynomials for the bases $u$ in \eqref{correctdefinition} contain 
 arbitrarily many terms? We introduce a construction of an infinite family of ladders of the desired form, in such a way so as to provide explicit, numerical 
 instances of the closed-form formulas given by our construction. This is inspired by how many of the most celebrated weight-2 ladders are such that the 
 minimal polynomial for the $u$-value in \eqref{correctdefinition} contains a greater number of terms than the base equations for classically known 
 ladders. In this direction, an especially renowned ladder relation is due to Bailey and Broadhurst and is such that 
\begin{align*}
 & \text{{\small $
 \text{Li}_{2}\left( u^{630} \right) 
 - 2 \text{Li}_{2}\left( u^{315} \right) 
 - 3 \text{Li}_{2}\left( u^{210} \right) 
 - 10 \text{Li}_{2}\left( u^{126} \right) 
 - 7 \text{Li}_{2}\left( u^{90} \right) $ }} \\
 & \text{{\small $ + 18 \text{Li}_{2}\left( u^{35} \right)
 + 84 \text{Li}_{2}\left( u^{15} \right) 
 + 90 \text{Li}_{2}\left( u^{14} \right) 
 - 4 \text{Li}_{2}\left( u^{9} \right) 
 + 339 \text{Li}_{2}\left( u^{8} \right) $ }} \\
 & \text{{\small $ + 45 \text{Li}_{2}\left( u^{7} \right) 
 + 265 \text{Li}_{2}\left( u^{6} \right) 
 - 273 \text{Li}_{2}\left( u^{5} \right) 
 - 678 \text{Li}_{2}\left( u^{4} \right) 
 - 1016 \text{Li}_{2}\left( u^{3} \right) $ }} \\
 & \text{{\small $ - 744 \text{Li}_{2}\left( u^{2} \right) 
 - 804 \text{Li}_{2}\left( u \right) 
 - 22050 \log^{2} u = -2003 \zeta(2) $ }} 
\end{align*}
 for the root $u$ of 
\begin{equation}\label{baseBB}
 x^{10}+x^9-x^7-x^6-x^5-x^4-x^3+x+1 = 0
\end{equation}
 in $(0, 1)$. 

\subsection{A motivating result}\label{subsectionmotivating}
 Let $u$ denote the unique positive root of 
\begin{equation}\label{deg10motivating}
 x^{10}+x^9+x^8+x^7+x^6+x^5-x^4-x^3-x^2-x-1 = 0, 
\end{equation}
 noting the resemblance to \eqref{baseBB}. Then our main construction, as in Section \ref{sectionConstruction} below, can be used to prove the valid 
 ladder relation such that 
\begin{equation}\label{displaymotivating}
 2 \text{Li}_{2}\left( u^6 \right) - 2 \text{Li}_{2}\left( u^5 \right) - \text{Li}_{2}\left( u \right) 
 - 25 \log^2\left( u \right) = -\zeta(2) 
\end{equation}
 holds, and this formula appears to be new. The resemblance between the minimal polynomials in \eqref{baseBB} and \eqref{deg10motivating} 
 motivates our infinite family of generalizations of the ladder relation in \eqref{displaymotivating}. This is further motivated by how ladders with 
 \eqref{baseBB} as the base equation for $u$ are a main object of study in the work of Cohen, Lewin, and Zagier \cite{CohenLewinZagier1992}. 

\section{Background and preliminaries}\label{sectionBackground}
 Among the most basic dilogarithm ladder identities that are known follow in a direct way from classically known functional equations for $\text{Li}_{2}$. 
 For example, for Legendre's chi-function $\chi_{v}(z) = \frac{ \text{Li}_{v}(z) - \text{Li}_{v}(-z) }{2} = \text{Li}_{v}(z) - 2^{-v} \text{Li}_{v}\left( 
 z^{2} \right)$ \cite[\S1.8]{Lewin1981}, a functional equation for $\chi_2$ known to Euler and Legendre provides the ladder relation 
\begin{equation*}
 \text{Li}_{2}\left( u^2 \right) - 4 \text{Li}_{2}\left( u \right) - \log^2\left( u \right) = -\frac{3 \zeta(2) }{2} 
\end{equation*}
 for $u = \sqrt{2} - 1$ \cite[p.\ 19]{Lewin1981}. The first development in the $20^{\text{th}}$ century in the construction of polylogarithm ladders 
 \cite[p.\ 4]{Lewin1991text} is due to Coxeter in 1935 \cite{Coxeter1935}, via relations as in 
\begin{equation*}
 \text{Li}_{2}\left( u^{6} \right) - 4 \text{Li}_{2}\left( u^{3} \right) 
 - 3 \text{Li}_{2}\left( u^{2} \right) + 6 \text{Li}_{2}\left( u \right) = \frac{7 \zeta(2) }{5} 
\end{equation*}
 for $u = \frac{\sqrt{5}-1}{2}$, and a next step foward was due to Watson in 1937 \cite{Watson1937}, via relations as in 
\begin{equation*}
\text{Li}_{2}\left( u^2 \right) - \text{Li}_{2}(u) + \log^2 u = -\frac{\zeta(2)}{7}
\end{equation*}
 for the positive root $u$ of $x^{3} + 2 x^2 - x - 1 = 0$. 
 
 What may be seen as a key discovery in the context of dilogarithm ladders is given by how Lewin found that powers of a \emph{fixed} variable provide 
 ways of examining the structure of $\text{Li}_{2}$ through functional equations such as Abel's 5-term relation \cite{Lewin1984}. 
 For example, from Abel's equation, it is immediate that 
\begin{multline*}
 \text{Li}_{2}\left( u^{n - p - q} \right) - \text{Li}_{2}\left( u^{n - p} \right) - \text{Li}_{2}\left( u^{n-q} \right) + \\ 
 \text{Li}_{2}\left( u^{p} \right) + \text{Li}_{2}\left( u^{q} \right) + p q \log^2 u = \zeta(2) 
\end{multline*}
 for 
\begin{equation}\label{base4term}
 u^{p} + u^{q} - u^{n} - 1 = 0, 
\end{equation}
 providing one of the most basic instantiations of dilogarithm ladders. The family of ladders associated with the base equation in \eqref{base4term} 
 may be seen as having the advantage of producing arbitrarily high-index ladders, along with the the disadvantage given by how $u$ is limited by it 
 being required to satisfy a $4$-term polynomial relation. This gives light to the research interest in the construction of dilogarithm ladders with base 
 equations with arbitrarily many terms. This leads us toward the preliminaries on functional equations for $\text{Li}_{2}$ presented below. 

 Of basic importance in the study of dilogarithm ladders is the 5-term functional equation due to Abel whereby 
\begin{multline*}
 \text{Li}_{2}\left( \frac{x}{1-y} \right) 
 + \text{Li}_{2}\left( \frac{y}{1-x} \right) 
 - \text{Li}_{2}\left( \frac{x}{1-x} \frac{y}{1-y} \right) 
 - \text{Li}_{2}(x) 
 - \text{Li}_{2}(y) \\ 
 = \log(1-x) \log(1-y), 
\end{multline*}
 and special cases of this relation provide the complement, the duplication, and the inversion formulas such that 
\begin{align}
 \text{Li}_{2}(z) + \text{Li}_{2}(1-z) & = \frac{\pi^2}{6} - \log z \log(1-z), \nonumber \\ 
 \text{Li}_{2}(z) + \text{Li}_{2}(-z) & = \frac{1}{2} \text{Li}_{2}\left( z^{2} \right), \nonumber \\
 \text{Li}_{2}(-z) + \text{Li}_{2}\left( -\frac{1}{z} \right) & = -\frac{\pi^2}{6} - \frac{1}{2} \log^2 z. \nonumber 
\end{align}
 Of key importance in our work is a duplication formula derived from the five-term relation for \emph{Rogers' dilogarithm} $L(z) := \text{Li}_{2}(z) + 
 \frac{1}{2} \log z \log(1-z)$, referring to the work of Kirillov \cite{Kirillov1995} for background on Rogers' dilogarithm, with $\text{Li}_{2}$ sometimes 
 referred to as \emph{Euler's dilogarithm}. This five-term relation is such that 
\begin{equation*}
 L(x) + L(y) = L(xy) + L\left( \frac{x(1-y)}{1 - xy} \right) + L\left( \frac{y(1-x)}{1 - x y} \right) 
\end{equation*}
 for $0 < x, y < 1$. 

\section{Construction}\label{sectionConstruction}
 A remarkable formula discovered by Khoi in 2014 \cite{Khoi2014} in a knot-theoretic context is such that 
\begin{equation}\label{Khoiformula}
 L\left( \frac{1}{\phi \left( \phi + \sqrt{\phi} \right) } \right) - L\left( \frac{\phi}{\phi + \sqrt{\phi}} \right) 
 = \frac{\pi^2}{20}, 
\end{equation}
 for the golden ratio $\phi = \frac{\sqrt{5} + 1}{2}$. The formula in \eqref{Khoiformula} was determined by Khoi by evaluating the Seifert volume 
 of manifolds produced via the application Dehn surgery on the figure-eight knot. It was left as an open problem by Khoi to determine whether or not the 
 formula in \eqref{Khoiformula} is \emph{accessible}, i.e., provable through applications of equivalent formulations 5-term and 2-variable relations for 
 $\text{Li}_{2}$. This open problem was solved by Lima in 2017 \cite{Lima2017}. In this direction, Lima applied Rogers' five-term relation to prove that 
\begin{equation}\label{Abelduplication}
 L(z) = L\left( \frac{1}{2-z} \right) + \frac{1}{2} L\left( 2z - z^2 \right) - \frac{\pi^2}{12}, 
\end{equation}
 so that Khoi's formula follows directly from \eqref{Abelduplication} by setting $z = \frac{1}{\phi(\phi + \sqrt{\phi})}$ and by evaluating one of the 
 resultant terms in closed form according to a classically known value for $\text{Li}_{2}\left( \frac{1}{\phi^2} \right)$. Since Khoi's open problem can be 
 solved directly as a special case of the identity in \eqref{Abelduplication} proved in 2017, this motivates our further applications of \eqref{Abelduplication}. 
 Based on extant literature on or related to polylogarithm ladders 
 \cite{AbouzahraLewin1985,AbouzahraLewin1986,AbouzahraLewinXiao1987,BorweinBroadhurstKamnitzer2001,CohenLewinZagier1992,Gangl2013,
GordonMcIntosh1997,Kirillov1995,Lewin1982,Lewin1993,Lewin1984,Lewin1986,Lewin1991text,Lewin1991Okayama}, the ladders we obtain from Theorems  
 \ref{theoremeven} and   \ref{oddtheorem} below appear to be new.  

\begin{theorem}\label{theoremeven}
 For a positive integer $n$, let $u$ denote the unique positive root of 
\begin{equation}\label{evendegree}
 x^{2n} + x^{2n-1} + \cdots + x^{n} - x^{n-1} - x^{n-2} - \cdots - x - 1 = 0. 
\end{equation}
 Then the ladder relation 
\begin{equation*}
 2 \text{\emph{Li}}_{2}\left( u^{n+1} \right) - 2 \text{\emph{Li}}_{2}\left( u^{n} \right) 
 - \text{\emph{Li}}_{2}\left( u \right) - n^2 \log^{2}\left( u \right) = -\zeta(2) 
\end{equation*}
 holds. 
\end{theorem}

\begin{proof}
  We begin by setting $z = u^{n + 1}$ in \eqref{Abelduplication},   for $u$ as specified.   We find that $\frac{1}{2-z} = u^n$, since  this is equivlaent to  
\begin{equation}\label{evenfirstarg}
  1 = 2 u^n - u \left( \sum_{k=0}^{n - 1} u^k - \sum_{k=n}^{2n-1} u^k \right),  
\end{equation}
  with \eqref{evenfirstarg} equivalent to 
\begin{equation}\label{evenequivarg1}
 1 = 2 u^n-\sum _{k=0}^{n-1} u^{k+1}+\sum _{k=n}^{2 n-2} u^{k+1}+\sum _{k=0}^{n-1} u^k-\sum _{k=n}^{2 n-1} u^k, 
\end{equation}
 so that a reindexing argument gives us that 
 \eqref{evenequivarg1} holds. 
 Similarly, we claim that 
 $2 z - z^2 = u$. This equality is equivalent to 
\begin{equation}\label{evenequivfurther}
 2 u^{n+1}-u^2 \left(\sum _{k=0}^{n-1} u^k-\sum _{k=n}^{2 n-1} u^k\right) = u, 
\end{equation}
 with \eqref{evenequivfurther} equivalent to 
\begin{equation}\label{2024120125042A2M2}
 2 u^{n+1}-\sum _{k=0}^{n-1} u^{k+2}+\sum _{k=n}^{2 n-3} u^{k+2}+(u+1) u^{2 n}
 = u 
\end{equation}
 for $n > 1$ and similarly for $n = 1$, with \eqref{2024120125042A2M2} equivalent to 
 $$ 2 u^{n+1}-\sum _{k=0}^{n-1} u^{k+2}+\sum _{k=n}^{2 n-3} u^{k+2}+\sum _{k=0}^{n-1} u^{k+1}+\sum _{k=0}^{n-1} u^k-\sum _{k=n}^{2 n-1} u^{k+1}-\sum
 _{k=n}^{2 n-1} u^k
 = 0 $$
 for $n > 1$ and similarly for $n = 1$. In turn, this is equivalent to 
\begin{equation}\label{evenfinalbeforeL}
 2 u^{n+1}-\sum _{k=0}^{n-1} u^{k+2}+\sum _{k=n}^{2 n-3} u^{k+2}+\sum _{k=0}^{n-1} u^{k+1}-\sum _{k=n}^{2 n-2} u^{k+1} = u 
\end{equation}
 for $n > 1$ and similarly for $n = 1$, so that a reindexing argument gives us that \eqref{evenfinalbeforeL} holds. So, from \eqref{Abelduplication}, for $z$ as specified, we find that 
 $ L\left( u^{n+1} \right) - 2 L\left( u^{n} \right) - L\left( u \right) = -\zeta(2)$. 
 By rewriting the left-hand side
 as a linear combination of 
 $\text{Li}_{2}$-expressions and logarithmic expressions , we find that it remains to prove 
 the cyclotomic equation such that 
\begin{equation}\label{20241201528AcpommM1A}
 \frac{(1-u) \left(1-u^n\right)^{2 n}}{\left(1-u^{n+1}\right)^{2 n+2}}
 = u^{2 n^2}. 
\end{equation}
 In this direction, we claim that 
 the relation 
\begin{equation}\label{202411301043PM1A}
 \frac{u \left(1-u^n\right)}{1-u^{n+1}} 
 = u^{n+1}, 
\end{equation}
 holds, with \eqref{202411301043PM1A} equivalent to 
\begin{equation}\label{202411301055PM4A}
 u - u^{n+1} = u^{n+1} - u^{2}
 \left( \sum_{k=0}^{n-1} u^{k} - \sum_{k=n}^{2n-1} u^k \right). 
\end{equation}
 Expanding \eqref{202411301055PM4A} and making use of the vanishing of \eqref{evendegree} for $x = u$, we see that 
 \eqref{202411301055PM4A} is equivalent to 
\begin{equation}\label{lastlabeleven}
 u - u^{n+1} 
 = u^{n+1}-\sum _{k=0}^{n-1} u^{k+2}+\sum _{k=n}^{2 n-3} u^{k+2}+\sum _{k=0}^{n-1} u^{k+1}-\sum _{k=n}^{2 n-2} u^{k+1}. 
\end{equation}
 Applying a reindexing, we find that \eqref{lastlabeleven} is equivalent to 
 $ -u^{n+1}
 = \sum _{k=n}^{2 n-3} u^{k+2}-\sum _{k=n}^{2 n-2} u^{k+1}$, 
 which holds according to a reindexing. 
\end{proof}

\begin{example}
 The $n = 5$ case provides the motivating result highlighted in Section \ref{subsectionmotivating}. 
\end{example}

\begin{theorem}\label{oddtheorem}
 For a positive integer $n$, let $u$ denote the unique positive root of $$ x^{2n+1} + x^{2n} + \cdots + x^{n} - x^{n-1} - x^{n-2} - \cdots - x - 1 = 
 0. $$ Then the ladder relation $$ 2 \text{\emph{Li}}_{2}\left( u^{n+2} \right) - 2 \text{\emph{Li}}_{2}\left( u^{n} \right) - 
 \text{\emph{Li}}_{2}\left( u^2 \right) -n^2 \log^2\left( u \right) = -\zeta(2) $$ holds. 
\end{theorem}

\begin{proof}
 We begin by setting $z = u^{n+2}$ in \eqref{Abelduplication}. We claim that $\frac{1}{2-z} = u^{n}$ and that $2 z-z^2 = u^2$, and reindexing 
 arguments as in our proof of Theorem \ref{theoremeven} 
 may be applied to prove this claim. 
 So, we find that \eqref{Abelduplication} provides us with 
 $ 2 L\left( u^{n+2} \right) - 2 L\left( u^{n} \right) - L\left( u^2 \right) = -\zeta(2)$. 
 Following a similar approach as in the proof for 
 Theorem \ref{theoremeven}, 
 it remains to prove the cyclotomic equation 
\begin{equation}\label{cyclotomicodd}
 \frac{\left(1-u^2\right) \left(1-u^n\right)^n}{\left(1-u^{n+2}\right)^{n+2}} 
 = u^{n^2}. 
\end{equation}
 In this direction, 
 the relation 
\begin{equation}\label{oddaftercyclotomic}
 \frac{u^2 \left(1-u^n\right)}{1-u^{n+2}} = u^{n+2}. 
\end{equation}
 is equivalent to 
 $$ u^2-u^{n+2}
 = u^{n+2}-\sum _{k=0}^{n-1} u^{k+3}+\sum _{k=n}^{2 n-3} u^{k+3}+\sum _{k=0}^{n-1} u^k-\sum _{k=n}^{2 n} u^k+u^{2 n+3}+u^{2 n+2} $$ 
 for $n \geq 2$ and similarly for the $n = 1$ case. We manipulate the right-hand side so that 
\begin{align*}
 & -\sum _{k=0}^{n-1} u^{k+3}+\sum _{k=0}^{n-1} u^k+u^{2 n+3}+u^{2 n+2}-u^n-u^{n + 1} \\ 
 = & -\sum _{k=0}^{n-1} u^{k+3}+\sum _{k=0}^{n-1} u^k+\sum _{k=0}^{n-1} u^{k+2}+\sum _{k=0}^{n-1} u^{k+1}-2 \sum _{k=n}^{2 n} u^{k+1}-u^n-u^{2 n+2} \\
 = & -\sum _{k=0}^{n-1} u^{k+3}+\sum _{k=0}^{n-1} u^k+\sum _{k=0}^{n-1} u^{k+2}-\sum _{k=n}^{2 n} u^{k+1}-u^n \\
 = & -\sum _{k=0}^{n-1} u^{k+3}+\sum _{k=0}^{n-1} u^{k+2}, 
\end{align*}
 giving us the desired evaluation as $ u^2-u^{n+2}$, and similarly for the $n = 1$ case. From \eqref{oddaftercyclotomic}, the cyclotomic relation in 
 \eqref{cyclotomicodd} is equivalent to $1-u^2 = \big(1-u^{n+2}\big)^2$, which may be shown via the reindexing approach above. 
\end{proof}

 According to Gordon and McIntosh \cite{GordonMcIntosh1997}, an \emph{$L$-algebraic number} $\theta$ is such that $0 < \theta < 1$ and satisfies a 
 vanishing condition of the form $ \sum_{k=0}^{n} c_{k} L\left( \theta^{k} \right) = 0 $ for integers $c_{k}$ that are not all equal to $0$, with Gordon 
 and McIntosh having obtained a number of values of degrees $3$ and $4$ satisfying the given conditions. This motivates our construction of 
 $L$-algebraic values of unboundedly high degree and with minimal polynomials with unboundedly many terms. In this direction, the quintic case of 
 Theorem \ref{oddtheorem} provides an analogue of results from Abouzahra and Lewin \cite{AbouzahraLewin1986}, as below. 

\subsection{The quintic case}
  Barrucand (cf.\ \cite{AbouzahraLewin1986}) considered the 
 problem as to whether or not there could be valid ladders with arguments given by the roots in $(0, 1)$ of 
\begin{align}
 u^5 - u^3 + u^2 + u - 1 & = 0 \ \text{and} \ \label{uquintic} \\ 
 v^{5} + v^4 - v^3 + v^2 - 1 & = 0. \label{vquintic} 
\end{align}
 By manipulating \eqref{uquintic} and \eqref{vquintic}
 so as to produce cyclotomic equations, 
 Abouzahra and Lewin obtained ladder relations such as 
\begin{multline*}
 2 \text{Li}_{2}\left( u^{9} \right) 
 + \text{Li}_{2}\left( u^{8} \right) - 2 \text{Li}_{2}\left( u^{6} \right) - 
 2 \text{Li}_{2}\left( u^4 \right) - \\ 
 2 \text{Li}_{2}\left( u^3 \right) - 4 \text{Li}_{2}\left( u \right) - 13 \log^2 u 
 = -4 \zeta(2). 
\end{multline*}
 Our construction can be used to obtain a dilogarithm ladder with a very similar 
 and inequivalent base equation, relative to \eqref{uquintic}--\eqref{vquintic}. 
 That is, by letting $w$ denote the positive root of $ w^5 + w^4 + w^3 + w^2 + w-1 = 0$, the ladder relation 
\begin{equation}\label{motivatingquintic}
 2 \text{Li}_{2}\left( w^5 \right) - \text{Li}_{2}\left( w^4 \right) - 2 \text{Li}_{2}(w) - \log^2(w) 
 = -\zeta(2) 
\end{equation}
 holds, with \eqref{motivatingquintic} 
 providing the $n = 2$ case of Theorem \ref{oddtheorem}. 

\section{Conclusion}
 A construction due to Abouzahra and Lewin in 1985 \cite{AbouzahraLewin1985} of dilogarithm ladders relied on Kummer's functional equations of two 
 variables, and explored the problem of producing ladders by increasing the weight of a known ladder, so as to experimentally obtain conjectured ladders, 
 of higher weights beyond what is permissible by Wechsung's theorem. A related approach was subsequently explored by Lewin in 1986 \cite{Lewin1986}. 
 Miscellaneous ladders obtained from 
 functional equations for $\text{Li}_{2}$ 
 were obtained by 
 Abouzahra and Lewin \cite[\S5]{Lewin1991text}, 
 as in with the application of 
 a 15-term functional equation 
 derived from the Clausen function and attributed to 
 both Rogers \cite{Rogers1907} and Kummer \cite{Kummer1840} 
 to obtain that 
 $$ 3 \text{Li}_{2}\left( u^4 \right) - 
 4 \text{Li}_{2}\left( u^3 \right) - 
 6 \text{Li}_{2}\left( u^2 \right) - 12 \text{Li}_{2}\left( u \right) 
 - 6 \log^{2}\left( u \right) = -7 \zeta(2) $$ 
 for $u$ as the 
 root of $x^4-2 x^3-2 x+1 = 0$ in $(0, 1)$. 
 This is in contrast to our systematic approach toward 
 constructing ladders with base equations with arbitrarily many terms, as suggested
 by the motivating example in Section 
 \ref{subsectionmotivating}. 

 The \emph{supernumary ladders} due to Abouzahra and Lewin \cite[\S5]{Lewin1991text} often rely on base equations following by setting the 
 arguments of Abel-type 5-term relations to powers of a fixed base, with the use of cyclotomic equations. For example, by writing $u$ in place of the 
 positive solution of 
\begin{equation}\label{fixed3term}
 u^{p} + u^q - 1 = 0
\end{equation}
 for the $p = 4$ and $q = 1$ case, an index-21 cyclotomic equation satisfied by $u$ led Abouzahra and Lewin to experimentally discover the 
 index-21 ladder relation 
\begin{multline*}
 2 \text{Li}_{2}\left( u^{21} \right) - 
 6 \text{Li}_{2}\left( u^{7} \right) - 
 14 \text{Li}_{2}\left( u^{6} \right) + \\ 
 21 \text{Li}_{2}\left( u^{2} \right) + 
 5 \text{Li}_{2}\left( u \right) + 
 31 \log^{2}\left( u \right) = 
 11 \zeta(2). 
\end{multline*}
 The minimal polynomial for $u$ is again of the form shown in \eqref{fixed3term}, and base equations of the form shown in \eqref{fixed3term} are 
 considered to be limited \cite[p.\ 11]{Lewin1991text}, recalling the minimal polynomial in \eqref{baseBB} associated with the Bailey--Broadhurst ladder 
 reproduced in Section \ref{sectionIntro}. This motivates our construction of dilogarithm ladders with base equations with arbitrarily many terms, and leads 
 leave the following for a future subject of research. How could our construction be extended with the use of cyclotomic equations, so as to obtain 
 dilogarithm relations with more nonzero $A_{r}$-coefficients in \eqref{correctdefinition}, using a similar approach as in the work of Abouzahra and Lewin 
 \cite[\S5]{Lewin1991text}? We also leave the following as an open problem. Numerically, we have discovered the conjectural relation 
\begin{multline*} 
 \text{{Li}}_2\left( \frac{1}{2 \phi ^2} -\frac{1}{2} \sqrt{-1-\frac{1}{\phi^2}} \right) 
 - \text{{Li}}_2\left(\frac{ 1 - \sqrt{(1-2 \phi) (1 + 2 \phi)} }{2} \right) = \\ 
 \frac{\ln ^2(\phi )}{2} + \frac{3 \pi \ln (\phi ) i }{5} + \frac{\pi ^2}{150}, 
\end{multline*}
 noting the resemblance and inequivalence to Khoi's formula in 
 \eqref{Khoiformula}. Could techniques and dilogarithm relations as in 
 Section \ref{sectionConstruction} be applied to prove this experimentally discovered result? 

\subsection*{Acknowledgements}
 The author is grateful to acknowledge support from a Killam Postdoctoral Fellowship from the Killam Trusts. 
 The author thanks 
 Kunle Adegoke and F\'{a}bio M.\ S.\ Lima for useful feedback.

 \

\noindent {\textsc{John M. Campbell}} 

\vspace{0.1in}

\noindent Department of Mathematics and Statistics

\noindent Dalhousie University

\noindent 6299 South St, Halifax, NS B3H 4R2

\noindent {\tt jmaxwellcampbell@gmail.com}

\end{document}